\documentclass[12pt]{extarticle} 
\usepackage[mathscr]{eucal}
\usepackage{amsfonts, amsthm, amsmath, latexsym, amssymb}
\usepackage{bbm}
\usepackage{tikz} 

\usepackage[a4paper, total={7in, 9in}]{geometry} 

\usepackage{pgfplots}
\usepackage{graphics}
\usepackage{float}

\usepackage[margin=1cm]{caption}
\usepackage{color}

\usepackage{enumerate} 
\usepackage{algorithm, algorithmic}

\usepackage[noadjust]{cite}
\usepackage{authblk}

\usepackage{setspace}

\newtheorem{theorem}{Theorem}[section]
\newtheorem{lemma}[theorem]{Lemma}

\newtheorem{proposition}[theorem]{Proposition}

\theoremstyle{definition}
\newtheorem{definition}[theorem]{Definition}
\newtheorem{example}[theorem]{Example}

\theoremstyle{remark}
\newtheorem{remark}[theorem]{Remark}

\numberwithin{equation}{section}

\newcommand{\cE}{\mathcal{E}}

\newcommand{\cN}{\mathcal{N}}

\newcommand{\cB}{\mathcal{B}}

\newcommand{\cP}{\mathcal{P}}

\newcommand{\bN}{\mathbb{N}}

\newcommand{\R}{\mathbb{R}}

\newcommand{\si}{\sigma }

\newcommand{\ga}{\gamma }
\newcommand{\Ga}{\Gamma }

\newcommand{\ones}{\mathbbm{1}}
\newcommand{\one}{\mathbf{1}}

\newcommand{\Mod}{\operatorname{Mod}}
\newcommand{\Dom}{\operatorname{Dom}}

\newcommand{\Adm}{\operatorname{Adm}}
\newcommand{\Ext}{\operatorname{Ext}}

\newcommand{\co}{\operatorname{conv}}

\newcommand{\bi}{\begin{itemize}}
	\newcommand{\ei}{\end{itemize}}

\newcommand{\BL}{\operatorname{BL}}
\newcommand{\lbr}{\left\{ }
\newcommand{\rbr}{\right\} }

\usepackage{comment}
\numberwithin{equation}{section}
\newcommand{\beqa}{\begin{eqnarray}}
	\newcommand{\eeqa}{\end{eqnarray}}
\newcommand{\beqan}{\begin{eqnarray*}}
	\newcommand{\eeqan}{\end{eqnarray*}}

\newcommand{\beq}{\begin{equation}}
	\newcommand{\eeq}{\end{equation}}
\newcommand{\beqn}{\begin{equation*}}
	\newcommand{\eeqn}{\end{equation*}}

\newcommand{\PP}{\mathscr{P}}

\newcommand{\bit}{\begin{itemize}}
	\newcommand{\eit}{\end{itemize}}

\newcommand{\Gahat}{\widehat{\Ga}}
\newcommand{\Gatil}{\widetilde{\Ga}}

\usepackage{hyperref}   
\usepackage{tikz}

\begin{document}
	
	\title{\bf Modulus of hypertrees\thanks{This material is based upon work supported by the National Science Foundation under Grant n. 2154032.}}
	\author[1]{Huy Truong}
	\author[1]{Pietro Poggi-Corradini}
	\affil[1]{\small Dept. of Mathematics, Kansas State University, Manhattan, KS 66506, USA.} 
	\date{}
	\maketitle
	\begin{abstract}
		Lorea \cite{Lo} and later Frank et al.~ \cite{Frank} generalized  graphic matroids to hypergraphic matroids. In \cite{Frank}, the authors introduced hypertrees as a generalization of spanning trees and proved  a form of  the theorem of Tutte \cite{on} and Nash-Williams \cite{edge-disjoint}. In \cite{basemodulus, pietrofairest, huyfulkerson}, the authors explored the modulus of the family of spanning trees in graphs and of the family of bases of matroids, and provided connections to the notions of strength and fractional arboricity. They also established Fulkerson duality for these families. 
		In this paper, we extend these results to hypertrees, and show that the modulus of hypertrees uncovers a hierarchical structure within arbitrary hypergraphs.
	\end{abstract}
	
	\vspace{0.1in}
	
	\section{Introduction}\label{sec:intro}
	
	The theory of blocking and anti-blocking polyhedra was introduced by 
	Fulkerson in \cite{fulkersonblocking, fulkersonanti}. This theory describes
	dual relationships between sets of constraints in optimization problems and 
	has significant applications in combinatorial optimization.  In \cite{huyfulkerson, chopraon}, it was shown that the {\it Fulkerson blocker family}, as defined in Definition \ref{def:blocker} below, of the family of spanning trees of a graph can be characterized by the family of feasible partitions of the vertex set. This result can be deduced from a theorem of Tutte 
	\cite{on} and of Nash-Williams \cite{edge-disjoint}. Furthermore, in \cite{basemodulus}, the authors extended this result to the family of bases of matroids. In this paper, we want to explore similar properties in the setting of hypergraphic matroids. 
	
	Hypergraphic matroids were studied first by Lorea \cite{Lo} and later
	by Frank et al \cite{Frank}. They can be seen as generalizations of graphic
	matroids. In \cite{Frank}, the authors generalized the notion of spanning trees to the one of hypertrees. Let $H = (V, E)$ be a hypergraph. For a non-empty set $X \subseteq V$ and a subset $F \subseteq E$, the set $F[X]$  represents the collection of hyperedges in $F$ that are entirely contained in $X$. The hypergraph $(X, E[X])$ is called the {\it subhypergraph induced by the vertex subset $X$} and is denoted by $H[X]$. 
	A set $F \subseteq E$ is called a {\it hyperforest} if
	$|F[X]| \le |X|-1$ for every non-empty subset $X \subseteq V$. A hyperforest $F$ is called a 
	{\it hypertree} of $H$ if $|F|=|V|-1$. In the special case where $H$ is a graph, $F \subseteq E$ is a hypertree if and only if $F$ is a spanning tree.

	Lorea \cite{Lo} demonstrated that the hyperforests of a hypergraph $H$ form the family of independent sets of a matroid, which is called a 
	\textit{hypergraphic matroid} and denoted by $M(H)$. 
	Later, Frank et al. \cite{Frank} continued the study of hypergraphic matroids 
	and established Theorem \ref{rank}, which is presented below. We give some notations for this theorem. Let $\PP = \{V_1, \ldots, V_k\}$ be a partition of the vertex set 
	$V$, and let $F \subseteq E$, we define $\delta_F(\PP)$ as the set of hyperedges in $F$ that intersect at least two parts in $\PP$. The notation $|\PP|$ represents the number of parts 
	in the partition $\PP$. For simplicity, we sometimes use $\delta(\PP)$ in 
	place of $\delta_E(\PP)$.
	
\begin{theorem}[\cite{Frank}]\label{rank}
	Given a hypergraph $H = (V,E)$, let $M(H)$ be the associated hypergraphic 
	matroid of $H$. For $F \subseteq E$, the rank $r(F)$ of $F$ is given by 
	the following formula:
	\[
	r(F) = \min \left\{ |V| - |\PP| + |\delta_F(\PP)| : \, \PP \text{ is a 
		partition of } V \right\}.
	\]
\end{theorem}

In \cite{Frank}, the authors extended a theorem of Tutte \cite{on} and 
Nash-Williams \cite{edge-disjoint} to give the maximum number of disjoint 
hypertrees contained in a hypergraph. The result is stated as follows.

\begin{theorem}[\cite{Frank}]\label{thm:nash}
	A hypergraph $H = (V,E)$ contains $k$ disjoint hypertrees if and only if 
	\begin{equation}\label{eq:k-p-c}
			|\delta(\PP)| \geq k (|\PP| - 1)
	\end{equation}
	holds for every partition $\PP$ of $V$.
\end{theorem}

Another fundamental property of hypergraphs concerns their decomposition into disjoint hyperforests. The {\it arboricity} of a graph is the minimum number of edge-disjoint forests into which the edge-set can be decomposed. Nash-Williams \cite{nash1964decomposition} gave a characterization of this number, and
Frank et al.~in \cite{Frank} extended this to hypergraphs as follows.

\begin{theorem}[\cite{Frank}]
	A hypergraph $H = (V,E)$ can be partitioned into $k$ disjoint hyperforests 
	if and only if, for every $X \subseteq V$, 
	\[
	|E[X]| \leq k (|X| - 1).
	\]
\end{theorem}

One of the main purposes of this paper is to generalize the Fulkerson duality 
between spanning trees and partitions of the vertex set to the context of hypergraphs utilizing Theorem \ref{thm:nash}. 
For a positive integer $k$, a hypergraph $H$ is said to be 
\textit{$k$-partition-connected} if $H$ satisfies (\ref{eq:k-p-c}). A 
$1$-partition-connected hypergraph is also referred to as a 
\textit{partition-connected} hypergraph. 
A hypergraph $H$ is said to be \textit{connected} if, for every non-empty proper 
subset $X$ of $V$, there exists a hyperedge of $H$ that intersects both $X$ and 
$V \setminus X$. It follows from the definition that partition-connected 
hypergraphs must be connected, but a connected hypergraph may not be 
partition-connected. 
Theorem \ref{thm:nash} above shows that $H$ is partition-connected if and only 
if $H$ contains a hypertree.
For a vector $x \in \mathbb{R}^E$ and $F \subseteq E$, we use $x(F)$ to denote 
$\sum_{e \in F} x(e)$. The indicator vector of $F$ is denoted by 
$\ones_F \in \mathbb{R}^{E}_{\geq 0}$. 
Building on Theorem \ref{thm:nash}, the authors in \cite{ba23al} demonstrated 
that given a connected hypergraph $H = (V,E)$, the polyhedron defined by the 
following system of partition inequalities for vectors $x \in \mathbb{R}^E$,
\begin{align}
	& x(\delta(\PP)) \geq |\PP| - 1, \quad \text{for all partitions $\PP$ of $V$,} 
	\label{in1} \\
	& x \geq 0, \label{in2}
\end{align}
has integral extreme points. We call this polyhedron the \textit{partition polyhedron} and denote it by $P = P(H) \subset \mathbb{R}_{\geq 0}^E$. Informally, this result shows that the Fulkerson blocker family of the family 
of partitions of the vertex set contains integral points. This motivates us to further explore this result as follows.

For every multiset $A$ from $E$, we denote $ V[A] \subset V$ as the set of vertices in $\bigcup_{e \in A} e$. The hypergraph $(V[A], A)$ is called the {\it subhypergraph induced by the hyperedge-multiset $A$} and denoted by $H[A]$.
	The indicator vector $x^{A} \in \mathbb{R}^E_{\geq 0}$ of $A$ is defined such that $x^{A}(e)$ represents the multiplicity of $e$ in $A$, see Example \ref{ex:ex1} for instance.
	We introduce the following definition.
	\begin{definition}\label{def:omega}
		The family of all indicator vectors of multisets $A$ from $E$ such that $H[A]$ is a hypertree with vertex set $V$ is called the {\it multi-tree} family of $H$ and is denoted by $\Omega = \Omega(H)$.
	\end{definition}
	\begin{example}\label{ex:ex1}
		Consider the hypergraph $H$ in Figure \ref{fig:example1}. A multiset $A$ from $E$ can be represented by its indicator vector $x^{A} = (x_1,x_2)$ of $A$, where each $x_i\in\bN$ is the number of copies of $e_i$ that are in $A$. Then, the multi-tree family $\Omega(H)$ contains 3 vectors: $(1,1),(2,0)$, and $(0,2)$.
	\end{example} 
	\begin{figure}[t]
		\centering
		\includegraphics[scale=0.5]{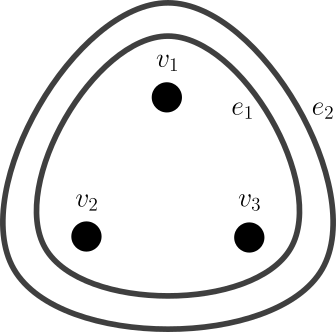}
		\caption{A hypergraph $H =(V,E)$ with $V =\lbr v_1,v_2,v_3 \rbr,$ $ E = \lbr e_1,e_2\rbr$, and $e_1 = e_2 =\lbr v_1,v_2,v_3 \rbr$.}
		\label{fig:example1} 
	\end{figure} 
 It was established in \cite{ba23al} that every extreme point of the partition polyhedron $P$ is a vector belonging to $\Omega(H)$. In the context of graphs, these extreme points are
	the indicator vectors of spanning trees. As a generalization, this paper investigates the Fulkerson duality between the 
	system of partition inequalities and the multi-tree family. Before stating our 
	result, we introduce some necessary notations.
	
	A partition $\PP=\{V_1, \ldots, V_k\}$ of $V$ is said to be {\it feasible} if each $V_i$ induces a connected subhypergraph $H[V_i]$ of $H$, for $i = 1,\dots,k$. For $v \in V$, the vertex-deleted subhypergraph $H \setminus v$ is obtained from $H$ by deleting $v$ from the vertex set, from all hyperedges containing $v$, and discarding any resulting empty hyperedges and loops. A hypergraph $H = (V,E)$ is said to be {\it vertex-biconnected}, if no single vertex deletion disconnects the hypergraph. 
	Let $f$ be a hyperedge of the hypergraph $H$. By $H/f$, we denote the hypergraph obtained from $H$ by contracting 
	the hyperedge $f$ into a new vertex $v_f$ and removing any resulting loops, if there are any. Specifically, $V(H/f) = (V(H) \setminus {f}) \cup \{v_f\}$, and a hyperedge $e'$ is in $ E(H/f)$, if and only if, either 
	$e' = e$ for some $e \in E(H)$ with $e \cap f = \emptyset$, or 
	$e' = (e \setminus f) \cup \{v_f\}$ for some $e \in E(H) \setminus \{f\}$ with $e \cap f \neq \emptyset$ and $e \cap (V(H) \setminus f) \neq \emptyset$. Let $F \subseteq E(H)$, then 
	$H/F$ is the hypergraph obtained from $H$ by contracting all hyperedges in $F$. If $S$ is a subhypergraph of $H$, then $H/S$ denotes $H/E(S)$.
	Let $\PP$ be a feasible partition, the shrunk hypergraph $H_{\PP}$ with respect to $\PP$ 
	is defined as the contraction $H / (E \setminus \delta(\PP))$. 
	See Definition \ref{def:blocker} for the notion of Fulkerson blocker families, the following theorem is our first main result. This theorem  provides a minimal description of the system of partition inequalities, see Remark \ref{re:minimal}.
	\begin{theorem}\label{thm:fulkerson}
		The Fulkerson blocker family $\widehat{\Omega}$ of $\Omega$ is the set of vectors $\frac{1}{|\PP|-1}\ones_{\delta(\PP)}$ for all feasible partitions $\PP$ whose shrunk hypergraphs $H_{\PP}$ are vertex-biconnected.
	\end{theorem}

	The next goal of this paper is to study modulus for the family of hypertrees and the multi-tree family of hypergraphs. The theory of {\it discrete modulus} has received significant attention in recent years \cite{modulus, pietrominimal, pietroblocking, pietrofairest, huyfulkerson, basemodulus}. This framework is a very flexible tool for measuring the richness of families of objects defined on a finite set. For instance, the modulus of the family of spanning trees in undirected graphs was investigated in \cite{pietrofairest, huyfulkerson}. More recently, \cite{basemodulus} studied the modulus of the family of bases of matroids, demonstrating its strong connection to two well-studied concepts in matroid theory: strength and fractional arboricity. 
	Following the concepts of strength and fractional arboricity in graphs, 
	we define the {\it strength} of a hypergraph $H$ as
	\begin{equation}\label{eq:S}
		S(H) := \min \left\{ \frac{|\delta(\PP)|}{|\PP| - 1} : \PP \text{ is a partition of } V, \ |\PP| \geq 2 \right\},
	\end{equation}
	and define the {\it fractional arboricity} of $H$ as
	\begin{equation}\label{eq:D}
		D(H) := \max \left\{ \frac{|E[X]|}{|X| - 1} : X \subseteq V, \ |X| \geq 2 \right\}.
	\end{equation}
For a hypergraph $H$ with given weights $\sigma \in \mathbb{R}^E_{>0}$ on the 
hyperedge set $E$, we generalize the strength of $H$ with respect to the 
weights $\sigma$ as follows:
\begin{equation}\label{eq:weighted-s}
	S_{\sigma}(H) = \min \left\{ \frac{\sigma(\delta(\PP))}{|\PP| - 1} : 
	\PP \text{ is a partition of } V, \ |\PP| \geq 2 \right\}.
\end{equation}

We aim to extend the concept of modulus for spanning trees to the context of 
hypergraphs; see Definition \ref{def:mod} for the definition of modulus. 
One immediate corollary of the Fulkerson duality in Theorem \ref{thm:fulkerson} 
is the relationship between $1$-modulus and strength, which is described as follows:
\begin{equation}
	\Mod_{1,\sigma}(\Omega(H)) = S_{\sigma}(H),
\end{equation}
for any set of weights $\sigma \in \mathbb{R}^E_{> 0}$, see Remark \ref{re:mod1-s}.
	
	Hence, in the unweighted case, the relationship simplifies to:
	\begin{equation}\label{eq:omega-s}
		\Mod_{1}(\Omega(H)) = S(H).
	\end{equation}
	
	Next, we consider the connection between the strength of a hypergraph and its 
	 family of hypertrees. Let $H$ be a partition-connected hypergraph, and let 
	$\Gamma = \Gamma(H)$ denote the family of all hypertrees of $H$, referred to as 
	the \textit{hypertree family} of $H$. 
	By definition of $\Omega(H)$, the set of indicator vectors of hypertrees in 
	$\Gamma(H)$ is contained in $\Omega(H)$. In general, 1-modulus of 
	the hypertree family does not necessarily coincide with the strength. In other words, we may have
	\begin{equation}\label{eq:s-not-equal}
		\Mod_{1,\sigma}(\Gamma(H)) \neq S_{\sigma}(H),
	\end{equation}
 for some weights $\si$, see Example \ref{ex:mod1-s} for instance. While (\ref{eq:s-not-equal}) shows that a general statement is not available, we provide a sufficient condition for equality in the following theorem. To our knowledge, this theorem cannot be deduced from Theorem \ref{thm:nash}, highlighting the importance of our proof.
	\begin{theorem}\label{mod-strength}
		Let $H$ be a partition-connected hypergraph. Then,
		\begin{equation}
			\Mod_1(\Ga(H)) = S(H).
		\end{equation}
	
	\end{theorem}

	Next, we shift our focus to $2$-modulus for the hypertree family $\Ga(H)$. We will establish connections 
	among the $2$-modulus of $\Ga(H)$, the strength, and the fractional arboricity of $H$ in Section \ref{sec:hypertree}. Specifically, following the notion of homogeneous graphs and homogeneous matroids in \cite{pietrofairest,huyfulkerson,basemodulus}, we introduce the following definition.
	\begin{definition}\label{def:hom1}
		Let $H$ be a partition-connected hypergraph, and let $\Ga = \Ga(H)$ be the hypertree family of $H$. Let $\widehat{\Ga}$ be the Fulkerson blocker family of $\Ga$, see Definition \ref{def:blocker}. The hypergraph $H$ is said to be {\it homogeneous} if the optimal density $\eta^*$ of $\Mod_2(\widehat{\Ga})$ is a constant vector.  
	\end{definition}
	We characterize homogeneous partition-connected hypergraphs in the following theorem.
	\begin{theorem}\label{thm:pc-hom}
		Let $H$ be a partition-connected hypergraph. Then $H$ is homogeneous if and only if its strength coincides with its fractional arboricity.
	\end{theorem}
	Furthermore, we show that the 2-modulus of the hypertree family reveals a 
	hierarchical structure within hypergraphs, which is related to strength and 
	fractional arboricity; see the \textit{Hypergraph Decomposition Process} 
	described in Section \ref{sec:hypertree}.
	
	A natural question arises: what happens when $H$ is not partition-connected? 
	Does it still reveal any hierarchical structure? The answer is yes. 
	Specifically, we will show that the modulus of the multi-tree family of any connected 
	hypergraph reveals a hierarchical structure within it, which is also related to 
	strength and fractional arboricity.
	Let $H = (V,E)$ be a connected hypergraph. Note that $H$ does not necessarily 
	contain any hypertree. A connection between the 2-modulus of the multi-tree family $\Omega(H)$ and 
	the hypertree family $\Gamma(H)$ is established in the following theorem.
	
	\begin{theorem}\label{thm:mod-mod}
		Let $H = (V,E)$ be a connected hypergraph. If $H$ is partition-connected, 
		then $\Mod_{2}(\widehat{\Omega})$ and $\Mod_{2}(\widehat{\Gamma})$ have the 
		same optimal density.
	\end{theorem}
	\begin{remark}
		By Theorem \ref{thm:mod-mod}, if $H$ is partition-connected, $\Mod_{2}(\widehat{\Omega})$ and $\Mod_{2}(\widehat{\Ga})$ reveal the same hierarchical structure of $H$, see the Hypergraph Decomposition Process described in Section \ref{sec:hypertree} and Section \ref{sec:omega}.
	\end{remark}
	The notion of homogeneity for partition-connected hypergraphs in Definition \ref{def:hom1} can be generalized to connected hypergraphs as follows.
	\begin{definition}\label{def:hom2}
		A connected hypergraph $H$ is said to be {\it homogeneous} if the optimal density $\tilde{\eta}^*$ of $\Mod_2(\widehat{\Omega}(H))$ is a constant vector. If $H$ is partition-connected, then by Theorem \ref{thm:mod-mod}, this definition coincides with Definition \ref{def:hom1}.
	\end{definition}
The following theorem extends Theorem \ref{thm:pc-hom} to the general case of 
connected hypergraphs.

\begin{theorem}\label{thm:c-hom}
	Let $H = (V,E)$ be a connected hypergraph. Then $H$ is homogeneous if and 
	only if its strength coincides with its fractional arboricity.
\end{theorem}

	\begin{figure}[t]
		\centering
		\includegraphics[scale=.6]{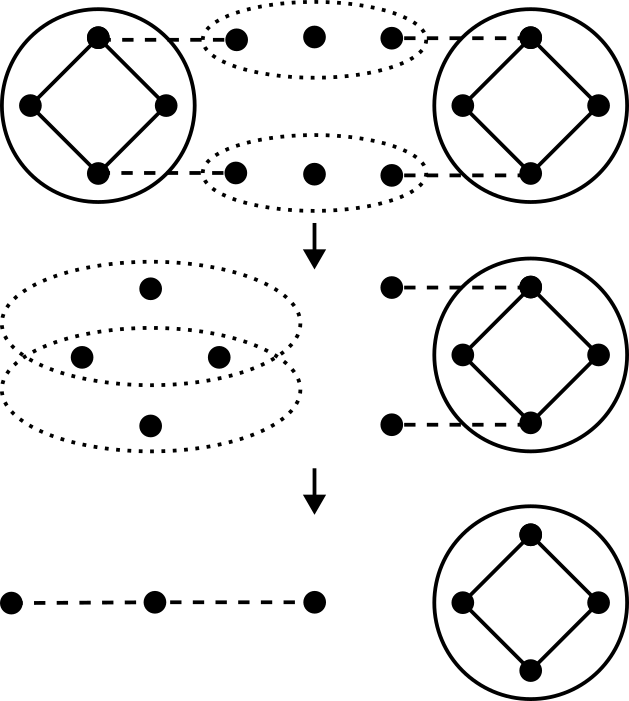}
		\caption{A hypergraph where hyperedges are styled according to the optimal density $\eta^*$. Solid edges are present with $\eta^* = \frac{3}{5}$, dashed edges with $\eta^* = 1$ and dotted edges with  $\eta^* = \frac{3}{2}$.}
		\label{fig:example2} 
	\end{figure}
	
	Finally, we present an example of $\Mod_{2}(\widehat{\Omega})$ revealing a hierarchical structure of a hypergraph $H$.
	\begin{example}\label{ex:fig2} Figure \ref{fig:example2} describes an example of the Hypergraph Decomposition Process discussed in Section \ref{sec:omega}. Consider a connected hypergraph $H$ in the top of Figure \ref{fig:example2}. Here, the optimal density $\eta^*$ of $\Mod_2(\widehat{\Omega}(H))$ takes three distinct values, indicated by the edge styles as described in the figure caption. 
		First, we remove the dotted hyperedges from $H$. As the result, we obtain two identical connected hypergraphs, denoted by $H_1$ and $H_2$, displayed in the middle right of the figure, and two isolated vertices. On the other hand, we shrink $H_1$ and $H_2$ in the original hypergraph $H$, the resulting shrunk graph is the hypergraph, denoted by $K$, displayed in the middle left. Notice that the hyperedge set of $H$ is the disjoint union of the hyperedge sets of $H_1,H_2$ and $K$. So, we say that $H$ decomposes into $H_1,H_2$, two isolated vertices and $K$. Next, for $H_1$ (similarly for $H_2$), we continue to decompose it by removing the dashed hyperedges from $H_1$. Then, we obtain a connected hypergraph displayed in the bottom right, denoted by $N$, and two isolated vertices. In $H_1$, we shrink $N$ to a single vertex to get the line hypergraph $L$ in the bottom left. The decomposition process stops here since all the created hypergraphs (isolated vertices, $K$, $N$ and $L$) are homogeneous.
	\end{example}
	\textbf{Organization of this paper:}
	\begin{itemize}
		\item Section \ref{sec:prel} introduces key definitions and notations.
		\item Section \ref{sec:Fulkerson} explores Fulkerson duality for the 
		system of partition inequalities and $\Omega(H)$. Specifically, we establish Theorem \ref{thm:fulkerson} which provides a minimal description of the system of partition inequalities, see Remark \ref{re:minimal}.
		
		\item Section \ref{sec:hypertree} focuses on the modulus of the family 
		of hypertrees in a hypergraph, referred to as the {\it hypertree modulus}.
		\item Section \ref{sec:omega} examines the modulus of $\Omega(H)$, 
		referred to as the {\it multi-tree modulus}.
	\end{itemize}
	
	\section{Preliminaries}\label{sec:prel}
	\subsection{Hypergraphs and hypertrees}\label{sec:prel1}
	
	A hypergraph $H$ is defined as a pair $(V, E)$, where $V$ is 
	the set of vertices and $E$ is a collection of nonempty subsets 
	of $V$, called hyperedges. The hyperedges in $E$ are not 
	necessarily distinct. A hyperedge consisting of a single 
	vertex is called a loop, and two hyperedges containing the 
	same vertices are called parallel hyperedges. In this paper, 
	we consider hypergraphs with no loops.

	Let $H = (V, E)$ be a hypergraph. For a subset $F \subseteq E$, 
	we denote $V[F] \subset V$ as the set of vertices in 
	$\bigcup_{e \in F} e$. The hypergraph $(V[F], F)$ is referred 
	to as the {\it subhypergraph induced by the hyperedge subset $F$} 
	and is denoted by $H[F]$. 
	Next, let $t$ be a positive integer. The hypergraph 
	$H^t = (V, E^t)$ is defined to have the same vertex set $V$, 
	where $E^t$ is constructed by replacing each hyperedge 
	$e \in E$ with a set of $t$ parallel hyperedges 
	$\{e^1, \dots, e^t\}$. The set of $t$ parallel hyperedges 
	$\{e^1, \dots, e^t\}$ in $E^t$ is called the 
	{\it parallel group} corresponding to the hyperedge $e$ 
	and is denoted by $g_e$. 
	Let $F$ be a subset of $E^t$. We define a multiset $F'$ 
	from $E$ generated by $F$, where the multiplicity of each 
	hyperedge $e \in E$ in $F'$ is the cardinality of 
	$F \cap g_e$.

	A hypergraph $H$ is said to be {\it forest-representable} or {\it wooded}, if it is possible to select two distinct vertices from each hyperedge of $H$ such that the chosen pairs, when considered as graph edges on $V$, form a forest. If the representing forest can be made connected, then $H$ is called {\it tree-representable}. Lov\'asz showed in \cite{lovasz70} that a hypergraph $H$ is a hyperforest if and only if it is wooded.
	
	For any nonempty subset $F \subseteq E(H)$, the density of $F$ is defined to be \[
	\theta_H(F) = \frac{|F|}{|V(H[F])| - 1}.
	\]
	We often use $\theta(H)$ to denote $\theta_{H}(E(H))$. Let $S(H)$ and $D(H)$ be the strength and  the fractional arboricity of $H$ defined as in (\ref{eq:S}) and (\ref{eq:D}), respectively. It follows immediately that for any loopless nontrivial hypergraph $H$, we have
	\[
	S(H) \leq \theta(H) \leq D(H).
	\]
	Moreover, it is also straightforward that
	\begin{equation*}
		S(H) = \min \{ \theta(H/F):F \subset E(H),|V(H/F)| \geq 2  \},
	\end{equation*}
	and
	\begin{equation*}
		D(H) = \max \{ \theta_H(F): F \subset E(H),|V(H[F])| \geq 2 \}.
	\end{equation*}

\subsection{Discrete modulus}
Let $E$ be a finite set with given weights 
$\sigma \in \mathbb{R}^E_{>0}$ assigned to each 
element $e$ in $E$. We say that $\Gamma$ is a family 
of objects in $E$ if, for each object $\gamma \in \Gamma$, 
there is an associated function  
$\mathcal{N}(\gamma, \cdot)^T: E \to \mathbb{R}_{\geq 0}$, 
which we interpret as a {\it usage vector} in 
$\mathbb{R}^E_{\geq 0}$. In other words, $\Gamma$ is 
associated with a $|\Gamma| \times |E|$ 
{\it usage matrix} $\mathcal{N}$. 
From now on, we assume that $\Gamma$ is nonempty and 
that each object $\gamma \in \Gamma$ uses at least one 
element in $E$ with a positive and finite amount.

A {\it density} $\rho$ is a vector in 
$\mathbb{R}^E_{\geq 0}$. For each $\gamma \in \Gamma$, 
we define the {\it total usage cost} 
$\ell_{\rho}(\gamma)$ of $\gamma$ with respect to $\rho$ as
\[
\ell_{\rho}(\gamma) = \sum\limits_{e \in E} 
\mathcal{N}(\gamma, e) \rho(e).
\]
A density $\rho \in \R^E_{\geq 0 }$ is  called {\it admissible} for $\Ga$, if for all $\ga\in \Ga$,
	$ \ell_{\rho}(\ga) \geq 1 .$
	The {\it admissible set} $\Adm(\Ga)$ of $\Ga$ is defined as the set of all admissible densities for $\Ga$,
	\begin{equation}\label{eq:adm-set}
		\Adm(\Ga) := \left\{ \rho \in \R^E_{\geq 0 }: \cN\rho \geq \one \right\}. 
	\end{equation} 
	Fix $1 \leq p < \infty$, the {\it energy} of the density  $\rho$ is defined as \[\cE_{p,\si}(\rho):=\sum\limits_{e \in E}\si(e)\rho(e)^p.\]
	
	\begin{definition}\label{def:mod}
		The {\it $p$-modulus} of $\Gamma$ is
		\begin{equation}\label{modp}
			\Mod_{p,\sigma}(\Gamma):= 
			\inf\limits_{\rho \in \Adm(\Gamma)} 
			\mathcal{E}_{p,\sigma}(\rho).
		\end{equation}
	\end{definition}
	
	Here is an example.
	
	\begin{example}\label{ex:mod1-s}
		Consider the hypergraph $H$ in Figure \ref{fig:example1}. 
		The hypertree family $\Gamma(H)$ contains only one 
		hypertree $\{e_1,e_2\}$ with the indicator vector $(1,1)$. 
		If we are given the weights $\sigma = (1,2)$, the strength 
		of $H$ is $ S_{\sigma}(H) = 3/2$, while 
		$\Mod_{1,\sigma}(\Gamma(H)) = 1$.
	\end{example} 
	
	When $\si$ is the vector of all ones, we omit $\si$ and write $\cE_{p}(\rho) := \cE_{p,\si}(\rho)$ and $\Mod_{p}(\Ga) :=\Mod_{p,\si}(\Ga)$.
	Next, let $K$ be a closed convex set in $\R^E_{\geq 0}$ such that $ \varnothing \subsetneq K \subsetneq \R^E_{\geq 0}$ and $K$ is {\it recessive}, meaning that $K +\R^E_{\geq 0} = K$. The {\it blocker} $\BL(K)$ of $K$ is defined as, \[\BL(K) = \lbr \eta \in \R^E_{\geq 0} : \eta^{T}\rho \geq 1, \forall \rho \in K \rbr. \] Note that the admissible set  $\Adm(\Ga)$ defined in (\ref{eq:adm-set}) is closed, convex in $ \R^E_{\geq 0}$ and recessive.
	We will routinely identify $\Ga$ with the set of its usage vectors $\left\{ \cN(\ga,\cdot )^T: \ga \in \Ga \right\}$ in $\R^E_{\geq 0}$, hence we can write $\Ga \subset \R^E_{\geq 0 }$. The {\it dominant} of $\Ga$ is defined as
	\[ \Dom(\Ga):= \co(\Ga) + \R^E_{\geq 0},\]
	where $\co(\Ga)$ denotes the convex hull of $\Ga$ in $\R^E$. We recall Fulkerson duality for modulus.
	\begin{definition}\label{def:blocker} Let $\Ga$ be a family of objects on a finite ground set $E$.
		The {\it Fulkerson blocker family} $\widehat{\Ga}$ of $\Ga$ is defined as the set of all the extreme points of $\Adm(\Ga)$:
		\[ \widehat{\Ga} := \Ext\left(\Adm(\Ga)\right) \subset  \R^E_{\geq 0}.\]
	\end{definition}
	Fulkerson blocker duality \cite{fulkersonblocking} states that
	\begin{equation}\label{eq:dom-adm-block-hat}
		\Dom(\widehat{\Ga})= \Adm(\Ga) = \BL(\Adm(\Gahat)),
	\end{equation}
	\begin{equation}\label{eq:dom-adm-block}
		\Dom(\Ga)= \Adm(\widehat{\Ga}) = \BL\left(\Adm(\Ga)\right).
	\end{equation}
	Moreover,  $\widehat{\Ga}$  has its own Fulkerson blocker family, and $\widehat{\widehat{\Ga}}  \subset \Ga$.
	\begin{definition}\label{def:fulkerson-dual}
		Let $\Ga$ and $\widetilde{\Ga}$ be two sets of vectors in $\R^E_{\geq 0}$. We say that $\Ga$ and $\widetilde{\Ga}$ are a {\it Fulkerson dual pair} (or $\Gatil$ is a {\it Fulkerson dual family} of $\Ga$) if \[\Adm(\Gatil) = \BL(\Adm(\Ga)).\] 
	\end{definition}
	Note that if $\Gatil$ is a Fulkerson dual family of $\Ga$, then $\Ga$ is also a Fulkerson dual family of $\Gatil$. Moreover, the Fulkerson blocker family $\Gahat$ is the smallest Fulkerson dual family of $\Ga$, meaning that $\Gahat \subset \Gatil.$
	
	When  $1<p < \infty$, let $q:=p/(p-1)$ be the Hölder conjugate exponent of $p$. For any set of weights $\si \in \R^E_{>0}$,  define the dual set of weights $\widetilde{\si}$ as $\widetilde{\si}(e):=\si(e)^{-\frac{q}{p}}$ for all $e\in E$. Let $\Gatil$ be a Fulkerson dual family of $\Ga$. Fulkerson duality for modulus  \cite[Theorem 3.7]{pietroblocking} states that
	\begin{equation}
		\Mod_{p,\si}(\Ga)^{\frac{1}{p}}\Mod_{q,\widetilde{\si}}(\widetilde{\Ga})^{\frac{1}{q}}=1.
	\end{equation}
	Moreover, the optimal $\rho^*$ of $\Mod_{p,\si}(\Ga) $ and the optimal $\eta^*$ of $\Mod_{q,\widetilde{\si}}(\widetilde{\Ga})$ always exist, are unique, and are related as follows,
	\begin{equation}\label{eq:weighted-eta-rho}
		\eta^{\ast}(e) = \frac{\si(e)\rho^{\ast}(e)^{p-1}}{\Mod_{p,\si}(\Ga)}, \quad \forall e\in E.
	\end{equation}
	When $p=2$, we have
	\begin{equation}\label{eq:mod2}
		\Mod_{2,\si}(\Ga)\Mod_{2,\si^{-1}}(\widetilde{\Ga})=1 \qquad\text{and}\qquad   \displaystyle \eta^{\ast}(e) = \frac{\si(e)}{\Mod_{2,\si}(\Ga)}\rho^{\ast}(e) \quad \forall  e\in E.
	\end{equation}
	According to the probabilistic interpretation of modulus \cite{pietrominimal}, we can express
	
	\begin{equation}\label{eq:modmeo}
		\Mod_2(\Ga)^{-1} = \min\limits_{\mu \in \cP(\Ga)} \mu^T\cN\cN^T\mu = \min\limits_{\eta \in \co(\Ga)} \sum\limits_{e \in E} \eta^2(e),
	\end{equation}
	where $\cP(\Ga)$ is the set of all probability mass functions in $\R^\Ga_{\geq 0}$ associated with $\Ga$.
	The middle optimization problem in (\ref{eq:modmeo}) is known as the {\it minimum expected overlap} (MEO) problem for $\Ga$, see \cite{pietrofairest,basemodulus}. And the optimization problem on the right side of (\ref{eq:modmeo}) is equivalent to $\Mod_2(\Gahat)$ since $\Adm(\Gahat) = \Dom(\Ga) = \co(\Ga) + \R^E_{\geq 0}$.

		\section{Fulkerson duality for hypertrees}\label{sec:Fulkerson}
		In this section, we prove Theorem \ref{thm:fulkerson}. 
		To start, we define the usage vectors of the family of 
		feasible partitions.
		
		\begin{definition}
			Let $H =(V,E)$ be a connected hypergraph. Let $\Phi$ 
			denote the family of all feasible partitions $\PP$ 
			of $H$, with the usage vectors defined as:
			\begin{equation}\label{usage2}
				\widetilde{\mathcal{N}}(\PP,\cdot)^T :=
				\frac{1}{|\PP|-1} \ones_{\delta(\PP)}.
			\end{equation}
		\end{definition}
		
		Then, the admissible set of $\Phi$ is given by:
		\begin{equation*}
			\Adm(\Phi) = \left\{ \eta \in \R^E_{\geq 0} :  \widetilde{\cN} \eta \geq \one \right\}.
		\end{equation*}
		Let $H =(V,E)$ be a connected hypergraph. Let $\Omega=\Omega(H)$ be the multi-tree family defined in Definition \ref{def:omega}. We want to study Fulkerson duality between two families  $\Phi$ and $\Omega$. We start with the following lemma.
		\begin{lemma}\label{lem:phi-P}
			The admissible set  $\Adm(\Phi)$ of $\Phi$ is equal to the partition polyhedron $P$ defined as in \eqref{in1}-\eqref{in2}.  
		\end{lemma}
		
		\begin{proof}
			By definition, we have that $ P \subset \Adm(\Phi)$. Let $\PP = \{V_1,\dots,V_k\}$ be a partition of $V$. For each $i =1,\dots,k$, if  $H[V_i]$ is not connected and has at least 2 connected components, we divide the class $V_i$ into smaller classes corresponding to the vertex sets of those connected components. This new partition is feasible, the number of classes is increased and the cut set $\delta(\PP)$ remains unchanged. Therefore $ \Adm(\Phi) \subset P$. Thus, we conclude that $ \Adm(\Phi) = P$.
		\end{proof}
		
		Since $\Adm(\Phi) = P$ where $P$ defined as in \eqref{in1}-\eqref{in2}, as mentioned in the introduction, we have that  \[\widehat{\Phi} \subset \Omega.\] For instance, the hypergraph in Figure \ref{fig:example2} has the partition polyhedron $P=\{(x_1, x_2) \, : \, x_1 + x_2 \ge 2, \, x \ge 0\}$. The extreme points of $P$ are $(2,0)$ and $(0,2)$. We use this fact to achieve the following result.
		\begin{proposition}\label{pro:dualpair}
			We have that \[\Dom(\Omega) = \Adm(\Phi).\] In other words, 
			$\Omega$ and $\Phi$ form a Fulkerson dual pair.
		\end{proposition}
		
		\begin{proof}
			In \cite{ba23al}, it was shown that every extreme point of 
			the partition polyhedron $P$ defined as in \eqref{in1}-\eqref{in2} 
			belongs to $\Omega(H)$. Hence, by Lemma \ref{lem:phi-P}, we have  
			$\widehat{\Phi} \subset \Omega$. This implies that 
			$\Adm(\Phi) = \Dom(\widehat{\Phi}) \subset \Dom(\Omega)$. 
			
			Let $x^A$ be an indicator vector in $\Omega$, where $A$ is a 
			multiset from $E$. By the definition of $\Omega$, the hypergraph 
			$H[A] = (V(H[A]),E(H[A]))$ is a hypertree with vertex set $V$, 
			and hence it is partition-connected. Thus, for any feasible 
			partition $\PP$ of $V$ in $\Phi(H)$,
			\[
			|\delta_{E(H[A])}(\PP)| \geq |\PP| -1.
			\]
			This is equivalent to 
			\[
			x^{A} \cdot \ones_{\delta_{E(H)}(\PP)} \geq |\PP| -1,
			\]
			for any feasible partition $\PP$ of $V$ in $\Phi(H)$.
			
			Therefore, $x^{A} \in \Adm(\Phi)$. Thus, 
			$\Dom(\Omega) \subset \Adm(\Phi)$. In conclusion, we have 
			$\Dom(\Omega) = \Adm(\Phi)$. 	
		\end{proof}
		
		By the definition of Fulkerson dual pairs, Proposition 
		\ref{pro:dualpair} implies that $\widehat{\Omega} \subset \Phi$. 
		Hence, the remaining task is to describe $\widehat{\Omega}$ 
		precisely. 
		
		Next, we provide the proof of Theorem \ref{thm:fulkerson}. 
		This proof follows a similar strategy from Fulkerson duality 
		for spanning trees in \cite{huyfulkerson}.

	\begin{proof}[Proof of Theorem \ref{thm:fulkerson}]
		Let $\Theta$ represent the family of all feasible 
		partitions $\PP$ of $H$ such that the shrunk 
		hypergraphs $H_{\PP}$ are vertex-biconnected. 
		
		The proof follows from the fact that 
		$\widehat{\Omega} \subset \Phi$ and the two lemmas 
		stated below. From Lemma \ref{lem:cl3}, we have 
		$\Phi \cap \widehat{\Omega} \subseteq \Theta$. 
		Furthermore, Lemma \ref{lem:cl2} establishes that 
		$\Theta \subset \widehat{\Omega}$.
	\end{proof}
	
	\begin{remark}\label{re:minimal}
		The Fulkerson blocker family $\widehat{\Omega}$ gives a 
		minimal inequality description $\Adm(\widehat{\Omega})$ 
		of the partition polyhedron $P$ defined as in 
		\eqref{in1}-\eqref{in2}. This is a common theme in 
		the study of various convex polytopes.
	\end{remark}
	
		\begin{lemma} \label{lem:cl3}
			We have  $ \Phi \cap \widehat{\Omega} \subseteq \Theta $.
		\end{lemma}
		\begin{proof} 
			Let $\PP \in \Phi \cap \widehat{\Omega}$. Suppose $\PP= \left\{ V_1,V_2,...,V_{k} \right\}$, each $H[V_i]$ shrinks to a vertex denoted by $v_i$ in  $H_{\PP}$. If $H_{\PP}$ is vertex-biconnected, then $P\in \Theta$. Assume that  $H_{\PP}$ is not vertex-biconnected. Then, there is a vertex $v_j$ such that the deletion of $v_j$ disconnects $H_{\PP}$ into two subhypergraphs $H_1$ and $H_2$ of $H_{\PP}$, each of them can have one or several connected components. Denote $A := \lbrace i: v_i \in \text{the vertex set of } H_1 \rbrace$. Denote $B := \lbrace i: v_i \in \text {the vertex set of } H_2 \rbrace$. Then,  $A,B$ and $\lbrace j\rbrace$ are disjoint and $A \cup B \cup \lbrace j\rbrace  = \lbrace 1,2\dots,k \rbrace$.
			Let \[  \displaystyle M := \left( \bigcup_{i\in B \cup \left\{ j\right\} } V_i \right)\subset V(H),\] we define a partition $\PP_1$ of the vertex set $V$ as follows: 
			\[  \PP_1 := \left\{ V_i: i \in A \right\}  \cup   \left\{ M \right\}.\] 
			Let \[  \displaystyle N := \left( \bigcup_{i\in A \cup \left\{ j\right\} } V_i  \right)\subset V(H) ,\]  we define a partition $\PP_2$ of the vertex set $V$ as follows: 
			\[  \PP_2 := \left\{ V_i: i \in B \right\}  \cup   \left\{ N \right\}. \]
			
			Since $v_j$ is connected with all the connected components of $H_1$ within $H_{\PP} \setminus H_2$, then the subhypergraph of $H_{\PP}$ induced by $A \cup \lbrace j \rbrace$  is connected. Therefore, the subhypergraph of $H$ induced by $N$ is connected. Hence, the partition $\PP_2$ is feasible. Similarly, the partition $\PP_1$ is feasible. Note that we have $|\PP_1| = |A|+1$ and $|\PP_2| = |B|+1$. Thus  $|\PP_1|+|\PP_2|=  |A|+1+|B|+1 =|\PP| +1$.  We also have that $\delta{\PP_1} \cup \delta{\PP_2} =\delta{\PP}$ and $\delta{\PP_1} \cap \delta{\PP_2} = \emptyset$. Therefore,

			\begin{equation}\label{convex}
				\frac{1}{|\PP|-1}\mathbbm{1}_{\delta(\PP)}= \frac{|\PP_1|-1}{|\PP|-1}.\frac{1}{|\PP_1|-1}\mathbbm{1}_{\delta(\PP_1)}+ \frac{|\PP_2|-1}{|\PP|-1}.\frac{1}{|\PP_2|-1}\mathbbm{1}_{\delta(\PP_2)}.
			\end{equation}
			Notice that $\frac{|\PP_1|-1}{|\PP|-1}+\frac{|\PP_2|-1}{|\PP|-1}=1$. So $\frac{1}{|\PP|-1}\mathbbm{1}_{\delta(\PP)}$ is a convex combination of  $\frac{1}{|\PP_1|-1}\mathbbm{1}_{\delta(\PP_1)}$ and  $\frac{1}{|\PP_2|-1}\mathbbm{1}_{\delta(\PP_2)}$. Since $\frac{1}{|\PP_1|-1}\mathbbm{1}_{\delta(\PP_1)}$,
			$\frac{1}{|\PP_2|-1}\mathbbm{1}_{\delta(\PP_2)}$$\in \Adm(\Omega)$ and  $\frac{1}{|\PP_1|-1}\mathbbm{1}_{\delta(\PP_1)} \neq$
			$\frac{1}{|\PP_2|-1}\mathbbm{1}_{\delta(\PP_2)}$, we have  $\frac{1}{|\PP|-1}\mathbbm{1}_{\delta(\PP)} \notin \widehat{\Omega}$, this is a  contradiction. Therefore, $H_{\PP}$ is vertex-biconnected and $\PP\in \Theta$.
		\end{proof}
		Before providing the next lemma, we give a remark as follows.
		\begin{remark}\label{rem:shrunk-pc}
			Let $\PP$ be a feasible partition of a hypergraph $H$. If $H$ is partition-connected, then the shrunk graph $H_P$ is also partition-connected.
		\end{remark}
		
		\begin{lemma} \label{lem:cl2}
			We have  $\Theta \subset \widehat{\Omega}$.
		\end{lemma}
		To prove this lemma, we first introduce additional notations and a useful theorem. Let $u,v$ be two vertices in $V$. A {\it walk} $W$ of length $k$ connecting $u$ and $v$ in $H$ is a sequence $v_0 e_1 v_1 e_2 v_2 \dots v_{k-1} e_k v_k$ 
		of vertices and hyperedges (possibly repeated) such that $v_0, v_1, \dots, v_k \in V$, 
		$e_1, \dots, e_k \in E$, $v_0 = u$, $v_k = v$, and for all $i = 1, 2, \dots, k$, the vertices 
		$v_{i-1}$ and $v_i$ are adjacent in $H$ via the edge $e_i$. Then, the hypergraph $H$ is connected if and only if there exists a walk connecting $u$ and $v$ for every pair of distinct vertices $u$ and $v$. For a walk $W = v_0 e_1 v_1 e_2 v_2 \dots v_{k-1} e_k v_k$ in $H$, if $k \geq 2$, $v_0 = v_k$, the vertices $v_0, v_1, \dots, v_{k-1}$ are pairwise distinct
		and the edges $e_1, \dots, e_k$ are also pairwise distinct, then $W$ is called a {\it cycle}. A {\it cut hyperedge} in $H$ is a hyperedge $e \in E$ such that its removal increases the number of connected components of $H$. A \textit{cut vertex} in a hypergraph $H = (V, E)$ with $|V| \geq 2$ is a vertex $v \in V$ such that its deletion increases the number of connected components of $H$.
		
		 We have the following theorem.

		\begin{theorem}\cite[Theorem 3.45]{connection15}\label{thm:commoncycle}
			Let $H = (V, E)$ be a connected hypergraph with $|V| \geq 2$, without hyperedges of cardinality less than 2, and without vertices of degree less than 2. Then the following are equivalent:
			\begin{enumerate}
				\item $H$ has no cut vertices and no cut hyperedges.
				\item Every pair of elements from $V \cup E$ lie on a common cycle.
				\item Every pair of vertices lie on a common cycle.
				\item Every pair of hyperedges lie on a common cycle.
			\end{enumerate}
		\end{theorem}
		
	Let $\Ga= \Ga(H)$ be the family of all hypertrees where 
	usage vectors are indicator vectors. Let $t = |V|$, and 
	recall the hypergraph $H^t = (V, E^t)$, where $E^t$ is 
	constructed by replacing each hyperedge $e$ in $E$ with 
	a set of $t$ parallel hyperedges $\{e^1,\dots,e^t\}$.  
	Let $\Ga(H^t)$ be the family of all hypertrees of $H^t$ 
	with indicator vectors as usage vectors. The following remark gives a connection between $\Ga(H^t)$ and $\Omega(H)$
	\begin{remark}\label{re:1-1}
		 Each hypertree 
		$\ga \in \Ga(H^t)$ corresponds to a multiset $\ga'$ of $E$ 
		with the indicator vector $x^{\ga'}$, where 
		$x^{\ga'} \in \Omega(H)$. Conversely, every vector in 
		$\Omega(H)$ corresponds to a hypertree in $\Ga(H^t)$. Hence, there is a one to one correspondence between the families $\Omega(H)$ and $\Ga(H^t)$.
	\end{remark}

	Using this theorem, we provide a useful lemma for the proof 
	of Lemma \ref{lem:cl2}.
	
		\begin{lemma}\label{lem:twovectors}
			Let $H$ be a vertex-biconnected hypergraph with $|V| \geq 2$ and $|E| \geq 2$. Let $e$ and $f$ be  two arbitrary distinct and possibly parallel hyperedges. Let $\Omega = \Omega(H)$ be the multi-tree family of $H$. Then there exists two vectors  $x_1,x_2$ in $\Omega$  such that $ x_1 - \ones_{e} + \ones_{f} =x_2$. 
		\end{lemma}
		\begin{proof}
			We consider the hypergraph $H^t$ with $t = |V|$. Since $H$ is vertex-biconnected, it follows that $H^t$ is also vertex-biconnected. By definition of  $H^t$, $H^t$ has no cut hyperedges and has no hyperedges with cardinality less than $2$, and no vertices with degree less than 2. 
			
			Let $e$ and $f$ be  two arbitrary distinct and possibly parallel hyperedges in $H$. The hyperedge $e$ corresponds the the group $g_e$ of $t$ parallel hyperedges $\{e^1,\dots,e^t\}$ in $H^t$ and $f$ corresponds to the group $g_f$ of $t$ parallel hyperedges $\{f^1,\dots,f^t\}$ in $H^t$. Applying Theorem \ref{thm:commoncycle}, we have that $e^1$ and $f^1$ lie on a common cycle $ \pi = v_0a_1v_1a_2v_2\dots a_kv_k$ where $v_k=v_0$ and $k \geq 2$. Let $A= \bigcup_{i=1,\dots,k} a_i$. Enumerate the set $A \setminus \lbr v_0,\dots, v_{k-1}\rbr$ as $\lbr x_1,\dots ,x_m \rbr$ and we extend the cycle $\pi$ as follows. To extend the cycle $\pi$, note that $x_1 \in A$, we pick an hyperedge $a_i$ such that $x_1 \in a_i$, and extend the cycle \[ \pi = v_0a_1v_1\dots v_{i-1}a_iv_i\dots a_kv_k\]
			to \[ v_0a_1v_1\dots v_{i-1}a_ix_1a'_iv_i\dots a_kv_k,\] where $a'_i$ is some hyperedge in the group of $t$ parallel hyperedges corresponding to $a_i$ but not included in the previous cycle. Note that the number of hyperedges in the cycle is at most $|V|=t$, and the previous cycle includes $e^1$ and $f^1$, which belong to different groups $g_{e}$ and $g_{f}$. Thus, there always exists such an $a'_i$, as each group has $t$ elements. We proceed to extend the cycle with $x_2,\dots,x_m$. In the end, this yields a cycle of size $|A|$. Thus, the set $C$ of $|A|$ hyperedges in the cycle forms a circuit in the graphic matroid $M(H^t)$, as we now explain. Indeed, note that the size of $C$ is equal to $|\bigcup_{c\in C}c|$, implying that $C$ is dependent. Furthermore, if any hyperedge $c$ is removed from $C$, the set $C \setminus \{c\}$ is tree-representable, with its representing tree being a path connecting all vertices in $|\bigcup_{c\in C}c|$ derived from the cycle. Therefore $C \setminus \{c\}$ is independent.
			Hence,  $C$ is a minimal depedent set, forming a circuit of the matroid $M(H^t)$.

			Given that $C$ is a circuit of the matroid $M(H^t)$ containing $e^1$ and $f^1$, it follows that $C -\lbr f^1 \rbr$ is a independent set. There exists a base $B_1$ of $M(H^t)$ such that $(C -\lbr f^1 \rbr) \in B_1$. Now, define the set $B_2 := (B_1 - \lbr e^1 \rbr) \cup \lbr f^1 \rbr$  and it is also a base of $M(H^t)$. The corresponding indicator vectors $x^{B'_1}$ and $x^{B'_2}$ in $\Omega$ of two multisets  $B'_1$ and $B'_2$ corresponding to $B_1$ and $B_2$ satisfy the following relation: 
			
			\[x^{B'_1}- \ones_{e} + \ones_{f} =x^{B'_2}.\] 
			This completes the proof.
		\end{proof}
		
		\begin{proof}[Proof for Lemma  \ref{lem:cl2}]
			 Let $\PP \in \Theta$, we want to show that 
			\[w :=\frac{1}{|\PP|-1}\mathbbm{1}_{\delta(\PP)} \in \Ext(\Adm(\Omega))=\widehat{\Omega}.\]
			Assume that there are two densities $\rho_1,\rho_2 \in \Adm(\Omega)$ such that 
			\[ w = \frac{1}{2}(\rho_1+\rho_2).\]
			For every hyperedge $e \in E\backslash \delta(\PP)$, we have \[\frac{1}{2}(\rho_1(e)+\rho_2(e))=w(e)= 0 \Rightarrow \rho_1(e)=\rho_2(e)=0.\] 
			
			Since $\PP \in \Theta$, we have that $H_{\PP}$ is vertex-biconnected. First, assume that $|\delta(\PP)| \geq 2$. Let $e$ and $f$ be  two arbitrary distinct hyperedges in $\delta(\PP)$. By Lemma \ref{lem:twovectors},  there exists two vectors  $x_1,x_2$ in $\Omega(H_{\PP})$  such that $ x_1 - \ones_{e} + \ones_{f} =x_2$. 
			
			For each $i = 1,\dots,|\PP|$, the hypergraph $H[V_i]$ is connected, there exists a vector $y_i$ in $\Omega(H[V_i])$. We concatenate all vectors $y_i$ with vector $x_1$ to form $z_1 \in \R^E_{\geq 0}$ and we have that $z_1\in \Omega(H)$. Similarly, concatenating all vectors $y_i$ with $x_2$ yields $z_2 \in \R^E_{\geq 0} $ and $ z_2 \in \Omega(H)$.
			
			Because  $\rho_1,\rho_2 \in \Adm(\Omega(H))$, we have
			\[\rho_1 \cdot z_1,\rho_1\cdot z_2 ,\rho_2\cdot z_1 ,\rho_2 \cdot z_2 \geq 1.\] 
			Since $\rho_1+\rho_2 = 2w$, we have  $(\rho_1+\rho_2)\cdot z_1 = 2w\cdot z_1 =  2w\cdot x_1 = 2$ and  $(\rho_1 +\rho_2)\cdot z_2 = 2w\cdot z_2 = 2w\cdot x_2 =  2$. This implies 
			\[\rho_1\cdot z_1 =\rho_2\cdot z_1 =\rho_1\cdot z_2 =\rho_2\cdot z_2 =1.\]
			Because  $\rho_1(e)=\rho_2(e)=0 $ for all $e \in E\backslash \delta(\PP)$, we obtain  
			\[\rho_1\cdot x_1= \rho_2\cdot x_1=\rho_1\cdot x_2= \rho_2\cdot x_2=1.\]
			Thus, 
			\begin{align*}
				0 = \rho_1\cdot (x_1-x_2) = \rho_1\cdot (\ones_e - \ones_f)
			\end{align*}
			Hence, since $e$ and $f$ were chosen arbitrarily, $\rho_1$ must be constant in $\delta(\PP)$. Similarly, $\rho_2$ is also constant in $\delta(\PP)$.
			Because $\rho_1\cdot x_1 =\rho_2\cdot x_1 =1$ and $ x_1(E) = |\PP|-1$ , we obtain that 
			
			\[\rho_1(e)= \rho_2(e) = \frac{1}{|\PP|-1}=w(e), \quad \forall e \in \delta(\PP).\]
			Therefore, $\rho_1=\rho_2=w $ and $w$ is an extreme point of $\Adm(\Omega(H))$.
			
			If $|\delta(\PP)|= 1$. By the same argument above, $w$ is an extreme point of $\Adm(\Omega(H))$.
			In conclusion, we have shown that $w \in \widehat{\Omega}$.
		\end{proof}
		
		If $H$ is partition-connected, equivalently, the family $\Ga(H)$ of all hypertrees is not empty. Then, the Fulkerson blocker family $\widehat{\Ga}(H)$ can be found by considering it as the base family of the associated hypergraphic matroids. Here is an example. The hypergraph in Figure \ref{fig:example1} has only one hypertree $\{e_1, e_2\}$ with the usage vector $(1, 1)$, and $\widehat{\Ga}(H)$ includes $(1, 0)$ and $(0, 1)$, which are not related to feasible partitions.

		

		\section{Modulus of the hypertree family}\label{sec:hypertree}
		
		In this section, we study the modulus of the hypertree 
		family of a partition-connected hypergraph. Specifically, 
		we analyze both 1-modulus and 2-modulus of this family, 
		and the main tool we utilize is the theory of the base 
		modulus in matroids.
		
		Let $H=(V,E)$ be a partition-connected hypergraph. This implies that $H$ has at least one hypertree. Let $\Ga = \Ga(H)$ be the hypertree family of $H$, where usage vectors are indicator functions. Let $\Omega = \Omega(H)$ be the multi-tree family of $H$. Note that $\Ga(H) \subset \Omega(H)$.  
		Let $\si \in \R^E_{>0}$ be a set of weights  assigned to the hyperedges in $E$. Our first objective is to analyze the 1-modulus of both the hypertree family and the multi-tree family. 
		\begin{remark}\label{re:mod1-s}
			By Proposition \ref{pro:dualpair}, two families $\Phi$ and $\Omega$ form a Fulkerson dual pair. Since the $1$-modulus is a linear program, at least one optimal density of $\Mod_{1,\si}(\Omega(H))$ must occur at an extreme point of $\Adm(\Omega(H))$. Hence, we have  \[\Mod_{1,\si}(\Omega(H)) = S_{\si}(H),\] for any set of weights $\sigma \in \R^E_{> 0}$.
		\end{remark}
	Let $M = M(H)$ be the hypergraphic matroid associated with $H$. An interesting question is how to interpret known results in the theory of modulus for bases of matroids on hypergraph. Let $r$ be the rank function of $M$. Then, the hypertree family $\Ga = \Ga(H)$ is the family of bases of $M$. Recall the weighted strength $s_{\si}(M)$ of the matroid $M$ with weights $\si \in \R^E_{> 0}$ defined as:
		\[s_{\si}(M) = \min \left\{ \frac{\si(X)}{r(E) - r(E\setminus X)} : X \subseteq E, r(E) > r(E \setminus X) \right\}.\]
		Here, $s_{\si}(M)$ is known as the strength of the matroid $M$, and it coincides with graph strength in the graphs setting. In \cite{basemodulus}, the authors show that \[s_{\si}(M) =  \Mod_{1,\si}(\Ga(H)).\]  Using Definition of $S_{\si}(H)$ in \ref{eq:weighted-s}, we may have
		\[S_{\si}(H) \neq  \Mod_{1,\si}(\Ga(H)).\] However, in the unweighted case, we can prove Theorem \ref{mod-strength}, which states that
			\begin{equation}
			\Mod_1(\Ga(H)) = S(H),
		\end{equation}
		 we will prove it after presenting results in 2-modulus.

		We now turn our attention to the 2-modulus of the hypertree family. 
		The following theorem is one of our main results.
		\begin{theorem}\label{main1}
			Let $H$ be a partition-connected hypergraph, and let 
			$\Ga = \Ga(H)$ be the hypertree family of $H$. Let 
			$\eta^*$ be the optimal density of  
			$\Mod_2(\widehat{\Ga})$, and let $E_{max}$ be the set 
			of elements where $\eta^*(e)$ attains the maximum 
			value of $\eta^*$. Then, $E_{max} = \delta_{H}(\PP)$ 
			for some feasible partition $\PP = \{V_1,\dots,V_k\}$ 
			of $V$ such that each subhypergraph $H[V_i]$ is 
			partition-connected, for $i=1,\dots,k$.
		\end{theorem}
		
		\begin{proof}
			Denote $ \eta^*_{max}$ as the maximum value of $\eta^*$. Let $M = M(H)$ be the hypergraphic matroid associated with $H$. By Theorem \cite[Theorem 4.1]{basemodulus}, we have that the set $E- E_{max}$ is closed in $M$, and
			\begin{align}
				\frac{1}{\eta^*_{max}}  = s(M),
			\end{align}
			where $s(M)$ is the strength of the unweighted matroid $M$.
			Moreover, Theorem \cite[Theorem 4.1]{basemodulus} shows that $E_{max}$ is optimal for the strength problem $s(M)$. Furthermore, $E_{max}$ is the maximal set that is optimal for $s(M)$ and this can be deduced from the proof of \cite[Theorem 4.1]{basemodulus} as follows. For any set $\emptyset \neq Y \not\subset E_{max}$, the set $Y$ is not optimal for $s(M)$ because
			\begin{align*}
				\eta^*_{max}&> \frac{\eta^*(Y)}{|Y|} & (\text{by definition of $\eta^*_{max}$})\\
				&\geq \frac{r(E) - r(E - Y)}{|Y|} & (\text{by \cite[Lemma 3.12 (i)]{basemodulus}}).
			\end{align*}
			
			Now, we remove $E_{max}$ from the hypergraph $H$ to obtain the subhypergraph $H_1 := H[E- E_{max}]$. Assume that $H_1$ has  $k$ connected components $L_1,\dots,L_k$ with $L_i=(V_i,C_i)$, $k \geq 1$. Let  $\PP$ be the partition of $V$ where each class of $\PP$ is the vertex set $V_i$ of $L_i$, for $i=1,\dots,k.$
			Then $\bigcup_{i =1,\dots,k} C_i = E - E_{max}$.  Note that the set $E - E_{max}$ is closed in $M$, and $E_{max} $ is the maximal set that is optimal for the strength problem $s(M)$, and
			
			\begin{align}
				1 \leq \frac{1}{\eta^*_{max}}= s(M) = \frac{|E_{max}|}{(|V|-1)- r(E- E_{max})}.
			\end{align} 
			Assume that $L_1$ has at least one hyperedge and is not partition-connected. Hence, $r(C_1) < |V_1| -1$. Using Theorem \ref{rank} for the rank function of the matroid $M(H)$ restricted to $C_1$, there exists a feasible partition $J = \{ J_1,J_2,\dots,J_l\}$ of $V_1$ such that
			\[ r(C_1) = |V_1|-|J| + |\delta_{L_1}(J)|.\] Thus,
			\begin{equation}\label{eq:rankL1}
				|\delta_{L_1}(J)|=|J|-|V_1|+r(C_1).
			\end{equation}
			Since $r(C_1) < |V_1| -1$, it follows that $|J|\geq 2$,  $\delta_{L_1}(J)$ is not empty, and $|\delta_{L_1}(J)| <|J|-1$.
			For the matroid strength problem $s(M)$,  we choose a candidate $X := E_{max}\cup \delta_{L_1}(J)$. For convenience, we also denote \[a = |E_{max}|,\quad b = |V| -1 - r(E \setminus E_{max}),\quad  \quad f =|\delta_{L_1}(J)| =  |J|-|V_1|+r(C_1).\]
			Because $\delta_{L_1}(J)$ is not empty, we obtain $E_{max} \subsetneq X$. Since the set $E_{max}$ is the maximal set that is optimal for the strength problem $s(M)$, we have
			\begin{align}\label{eq:strict-ine}
				1 \leq s(M)= \frac{a}{b} < \frac{|X|}{|V|-1 - r(E \setminus X)} = \frac{a+f}{b + r(E \setminus E_{max}) - r(E \setminus X)}.
			\end{align}
			Now, we aim to prove 
			\begin{equation}\label{eq:diff-f}
				r(E \setminus E_{max}) - r(E \setminus X) \geq |J|-|V_1|+r(C_1) = f.
			\end{equation}
			Indeed, we have
			\begin{align*}
				&r(E \setminus E_{max}) - r(E \setminus X)  - f= r(C_1) - r(C_1 \setminus \delta_{L_1}(J)) - f\\
				&=r(C_1) - r(C_1 \setminus \delta_{L_1}(J)) - (|J|-|V_1|+r(C_1)) \\
				&= |V_1| - r(C_1 \setminus \delta_{L_1}(J)) - |J|\\
				& =  (|J_1| - r(E(L_1[J_1]))+\dots+(|J_l| -r(E(L_1[J_l])) - |J|\\
				& \geq 1 +\dots+ 1 -|J| = |J| -|J| =  0.
			\end{align*}
			Therefore, (\ref{eq:strict-ine}) implies that
			\[\frac{a}{b} < \frac{a+f}{b+f}.  \] 
			This is equivalent to $a < b$, which contradicts  $a/b \geq 1$.  Therefore, $\delta_{L_1}(J)$ must be empty, $|J|= 1$, and $L_1$ is partition-connected.
			Similarly, every other connected component $L_i$ at least one hyperedge of the subhypergraph $H[E- E_{max}]$ is partition-connected.
			
			To conclude, we want to show that $E_{max} = \delta_{H}(\PP)$. If there exists a hyperedge $e \in E_{max}$ such that it lies in $H[V_i]$, then adding $e$ to $E(L_i)$ does not change the rank of $E(L_i)$ since $L_i$ is partition-connected. This contradicts the fact that $E - E_{max}$ is closed in $M$. Therefore, we must have $E_{max} \subset \delta_{H}(\PP)$. Let $x \in \delta_{H}(\PP)$. Then $x \in E_{max}$; because otherwise, $x$ would connect two connected components of $H_1$, which is a contradiction. Thus, we conclude that $E_{max} = \delta_{H}(\PP)$.
		\end{proof}

		Next, we will use Theorem \ref{main1} to prove Theorem \ref{mod-strength}. Before doing so, we make some remarks.
		
		\begin{remark}\label{re:s-neq-s}
			If $H$ is not partition-connected, then $H$ has no hypertree and every base of the hypergraphic matroid $M(H)$ has rank less than $|V|-1$. In general, we have $s(M(H)) \neq S(H)$. Here is an example, consider the hypergraph $H =(V,E )$ with $V =\lbr v_1,v_2,v_3 \rbr, E = \lbr e\rbr$ and $e = \lbr v_1,v_2,v_3 \rbr$. Then,  $ S(H) =  1/2$ while $s(M(H)) = 1$.
		\end{remark}
		\begin{remark}\label{mod-strength-mat}
			Let $H$ be a partition-connected hypergraph. The hypertree family $\Ga = \Ga(H)$ is the family of bases of the matroid $M(H)$.  Let $\eta^*$ be the optimal density of  $\Mod_2(\widehat{\Ga})$ and denote $ \eta^*_{max}$ as the maximum value of $\eta^*$. By the theory of the base modulus for matroids in \cite{basemodulus}, we have \[\Mod_1(\Ga(H)) = s(M(H)) = \frac{1}{\eta_{max}^*}.\] 
		\end{remark}
		Now, we present the proof of Theorem \ref{mod-strength}.
		\begin{proof}[Proof of Theorem \ref{mod-strength}]
			Let $M = M(H)$ be the hypergraphic matroid associated with $H$. By Remark \ref{mod-strength-mat}, we have $\Mod_1(\Ga(H)) = s(M)$. By Theorem \ref{main1}, $E_{max} = \delta_{H}(\PP)$  for some feasible partition $\PP = \{V_1,\dots,V_k\}$ of $V$ such that each subhypergraph $H[V_i]$ is partition-connected, for $i=1,\dots,k$. Note that $E_{max} $  is the maximal set that is optimal for the strength problem $s(M)$, then
			\begin{align*}
				s(M(H)) & = \frac{|E_{max}|}{r(E)-r(E- E_{max})}=\frac{|\delta_H(\PP)|}{|\PP|-1}\geq S(H).
			\end{align*}
			For any $X \subset E$ such that $ r(E)-r(X) >0$, assume that hypergraph $H[X]$ has $k$ connected components $H_i = (V_i,X_i)$, for $i=1,\dots,k$. Then \[|V(H/X)| = |V(H)| - \sum\limits_{i =1}^k(|V_i|-1) \leq |V(H)| - \sum\limits_{i =1}^k r(X_i) = |V(H)|- r(X) = r(E)+1-r(X).\] 
			Hence,
			\begin{align*}
				s(M) \leq \frac{|E-X|}{r(E)-r(X)} \leq \frac{|E-X|}{V(H/X)-1}.
			\end{align*}
			Thus, $s(M) \leq S(H)$. In conclusion, we have $s(M) = S(H)$. Therefore, $\Mod_1(\Ga(H)) = s(M)=S(H).$
		\end{proof}

		The next theorem describes the serial rule for the modulus of the hypertree family $\Ga(H)$, this theorem is motivated by the serial rule for the modulus for bases of matroids, see \cite[Theorem 3.19]{basemodulus}. 
		\begin{theorem}\label{thm:emax-serial}
			Let $H$ be a partition-connected hypergraph. Let $\Ga =\Ga(H)$ be the family of hypertrees of $H$. Assume that the optimal solution $\eta^*$ of $\Mod_2(\widehat{\Ga})$ is not a constant vector and let $E_{max}$ be the set of hyperedges $e$ where $\eta^*(e)$ attains the maximum value of $\eta^*$.  Let $H/(E-E_{max})$ be the hypergraph obtained from $H$ by contracting $E-E_{max}$. Let the hypergraph $H[E- E_{max}]$ have $k$ connected components $L_1,\dots,L_k$ with $L_i=(V_i,C_i)$, for $i=1,\dots,k$. Then, 
			\bi 
			\item[ (i)] The hypergraph $H/(E-E_{max})$ and every hypergraph $L_i$ are partition-connected, for $i=1,\dots,k$.
			\item[ (ii)] The modulus problem $\Mod_2 (\widehat{\Ga}(H))$ splits as follows:\begin{equation}
				\Mod_2 (\widehat{\Ga}(H)) = \Mod_2 (\widehat{\Ga}(L_1)) + \Mod_2 (\widehat{\Ga}(L_2)) +\dots + \Mod_2 (\widehat{\Ga}(L_k))+ \Mod_2 (\widehat{\Ga}(H/(E-E_{max})));
			\end{equation}
			\item[ (iii)] The restriction of $\eta^*$ onto $C_i$ is optimal for $\Mod_2(\widehat{\Ga}(L_i)) $ for $i = 1,2,\dots,k$ and the  restriction of $\eta^*$ onto $E_{max}$  is optimal for $\Mod_2(\widehat{\Ga}(H/(E-E_{max})))$. 
			\ei
		\end{theorem}
	
		\begin{proof}
			By Remark \ref{rem:shrunk-pc} and Theorem \ref{main1}, we obtain (i). Denote the base family of a matroid $M$ by $\cB(M)$. We recall the serial for the base modulus $\cB(M) = \Ga(H)$ of $M(H)$. Let $\eta^*$ be the optimal density for $\Mod_2(\widehat{\Ga}(H))$. Let $M | (E - E_{max})$ be the restriction of $M(H)$ to $E - E_{max}$. Let $M / (E - E_{max})$ be the contraction of $E -E_{max}$ in $M$.
			By \cite[Theorems 3.19 and 4.1]{basemodulus}, we have 
			\begin{equation}
				\Mod_2 (\widehat{\cB}(M)) = \Mod_2 (\widehat{\cB}(M | (E - E_{max})) + \Mod_2 (\widehat{\cB}(M /(E-E_{max}))).
			\end{equation}
			Moreover, the restriction of $\eta^*$ onto $E - E_{max}$ is optimal for $\Mod_2(\widehat{\cB}(M | (E - E_{max})))$, and the  restriction of $\eta^*$ onto $E_{max}$  is optimal for $\Mod_2(\widehat{\cB}(M / (E - E_{max})))$. 
			Note that the matroid $M / (E - E_{max}))$ is the hypergraphic matroid associated with $H/(E - E_{max})$ because $E_{max} = \delta_{H}(\PP)$ and (i), where $\PP$ is the partition of $V$ where each class of $\PP$ is the vertex set $V_i$ of $L_i$ for $i=1,\dots,k$. We also have that the restriction $M | (E - E_{max})$ is the hypergraphic matroid associated with $H[E- E_{max}]$. Therefore, we achieve (ii) and (iii).
		\end{proof}
	Theorem \ref{thm:emax-serial} uncovers a hierarchical structure within hypergraphs, this is described as follows.

		\textbf{Hypergraph Decomposition Process (HDP):} Let $H$ be a partition-connected hypergraph that is not homogeneous. First, using Theorem \ref{thm:emax-serial}, we decompose $H$ into $H/(E - E_{max})$ and $k$ partition-connected hypergraphs $L_i$ as defined in Theorem \ref{thm:emax-serial}. It follows that $H/(E - E_{max})$ is homogeneous. By Theorem \ref{mod-strength} and Remark \ref{mod-strength-mat}, we have \[\frac{1}{S(H)}=\eta^*_{max} = \frac{1}{S(H/(E - E_{max}))}.\] We proceed to further decompose each hypergraph $L_i$ using the maximum value of $\eta^*$, continuing this decomposition process until we only have a collection of homogeneous hypergraphs. Let $E_{min}$ be the set of elements where $\eta^*(e)$ attains the minimum value of $\eta^*$.
		At the end of this process, we obtain a collection of connected components of $H[E_{min}]$ that are vertex-induced, partition-connected, and homogeneous. We call each of these 
		a {\it homogeneous core}.  Furthermore, each homogeneous core $K$ of $H[E_{min}]$ satisfies
		\[\eta^*_{min} = \frac{1}{S(K)}.\]
		This process yields a collection of subhypergraphs and shrunk hypergraphs which correspond to a hierarchical structure of $H$, see Example \ref{ex:fig2} for an example. 
		
		Next, similar to Theorem \ref{mod-strength}, we have a corresponding theorem for fractional arboricity. Recall the fractional arboricity of the hypergraphic matroid $M=M(H)$ which is defined as:
		\[D(M) = \max \left\{ \frac{|X|}{r(X)} : X \subseteq E, r(X)>0 \right\} .\]
		\begin{theorem}\label{thm:mod:f}
			Let $H$ be a partition-connected hypergraph. Let $M = M(H)$ be the hypergraphic matroid associated with $H$.  Then,
			$D(M(H))= D(H).$
		\end{theorem}
		\begin{proof}
			We have that the hypertree family $\Ga = \Ga(H)$ of $H$ is the family of bases of the matroid $M(H)$. Let $\eta^*$ be the optimal density of  $\Mod_2(\widehat{\Ga})$. Let $E_{min}$ be the set of elements where $\eta^*(e)$ attains the minimum value $\eta^*_{min}$ of $\eta^*$.
			By the theory of the base modulus for matroids in \cite{basemodulus}, we have \[ D(M(H)) = \frac{1}{\eta_{min}^*}.\]

			By the Hypergraph Decomposition Process of $H$, we have that every connected component of $H[E_{min}]$ is vertex-induced, partition-connected, and homogeneous. Let $K$  be a connected component among these components. Then,
			\begin{align*}
				D(M(H)) =  \frac{|E_{min}|}{r(E_{min})} = \frac{1}{\eta^*_{min}} =\frac{|E(K)|}{r(M(K))} = \frac{|E(K)|}{V(K)-1} \leq D(H).
			\end{align*}
			On the other hand, for all $X \subseteq E$ such that $r(X)>0$ , we have
			\begin{align*}
				D(M(H)) \geq   \frac{|X|}{r(X)} \geq \frac{|X|}{V(H[X])-1}.
			\end{align*}
			Thus, this yields $ D(M(H)) \geq D(H).$ In conclusion,  \[\frac{1}{\eta_{min}^*}=D(M(H))= D(H).\]
		\end{proof}
		Next, we prove Theorem \ref{thm:pc-hom}.
		\begin{proof}[Proof for Theorem \ref{thm:pc-hom}]
			Let  $\Ga = \Ga(H)$ be the hypertree family of $H$. Let $\eta^*$ be the optimal density of  $\Mod_2(\widehat{\Ga})$. Let $M=M(H)$ be the hypergraphic matroid associated with $H$. By the theory of modulus for bases of matroids, $\eta^*$ is a constant vector if and only the strength of $M$ coincides with the fractional arboricity of $M$. By Theorem \ref{mod-strength} and Theorem \ref{thm:mod:f}, the proof is completed.
		\end{proof}
		\begin{remark}\label{re:reverse-thm}
			We also have a reverse version of the serial rule for the modulus of hypertrees in Theorem \ref{thm:emax-serial}, where we can shrink any homegeneous core $K$ and splits the modulus problem as follows:
			\begin{equation}
				\Mod_2 (\widehat{\Ga}(H)) = \Mod_2 (\widehat{\Ga}(K)) + \Mod_2 (\widehat{\Ga}(H/K)).
			\end{equation}
		\end{remark}
		Finally, based on Remark \ref{re:reverse-thm}, we present a reverse version for the Hypergraph Decomposition Process.
		
		\textbf{Hypergraph Shrinking Process (HSP):} The hierarchical structure established by the Hypergraph Decomposition Process can also be uncovered by reversing the process as follows. Note that every homogeneous core $K$ of $H$ in $H[E_{min}]$ is vertex-induced, partition-connected, homogeneous, and satisfies
		\[\frac{1}{D(H)}=\eta^*_{min} = \frac{1}{D(K)}= \frac{1}{S(K)}.\]
		In the first step, we shrink every homogeneous core $K$ of $H$ to get a shrunk hypergraph $H_1$ which is partition-connected. We proceed shrinking with the same strategy for $H_1$ using the minimum value of $\eta^*$, and continue this process until we only have a shrunk hypergraph, denoted by $L$,  which is homogeneous. It follows that $L$ satisfies
		\[\frac{1}{S(H)}=\eta^*_{max} = \frac{1}{S(L)}=\frac{1}{D(L)}.\]
		This process reveals the same collection of subhypergraph and shrunk hypergraphs as in the Hypergraph Decomposition Process.

		\section{Modulus for the multi-tree family}\label{sec:omega}
		In this section, we answer the following questions: 
		What happens when $H$ is not partition-connected? 
		Does it still have any hierarchical structure? 
		Let $H = (V,E)$ be a connected hypergraph that is not 
		necessarily partition-connected. The idea is to build a 
		hypergraph $H^t = (V, E^t)$ from $H$, where $E^t$ is 
		obtained by replacing each hyperedge $e$ in $E$ with a 
		set of $t$ parallel hyperedges. Let $\Ga(H^t)$ be the hypertree family of the hypergraph $H^t$ where usage vectors are indicator vectors. By Remark \ref{re:1-1}, there is a one to one correspondence between the families $\Omega(H)$ and $\Ga(H^t)$.
		Then, we will use results from hypertree modulus in 
		the hypergraphs $H^t$ to study the multi-tree modulus 
		of $\Omega(H)$. The main tool for this strategy is 
		presented in the following section.

		\subsection{Symmetry property for modulus}\label{sec:sym}

		In this section, we provide a symmetry property for modulus.
		 Given a ground set $E$ with weights $ \si \in \R^E_{>0}$. A
		\textit{symmetry} is an involution 
		$Q : E \to E$ such that
		\begin{equation}\label{eq:symmetry}
				\sigma(Q(e)) = \sigma(e) \quad \forall e \in E.
		\end{equation} 
	
		Given a symmetry $Q$, we say that a family of objects $\Gamma$ is $Q$-invariant if for every object 
		$\gamma_1 \in \Gamma$, there is an object $\gamma_2 \in \Gamma$ such that
		\begin{equation}\label{eq:invariant}
			\cN(\gamma_2, Q(e)) = \cN(\gamma_1, e), \quad \forall e \in E.
		\end{equation}
		We set 
		\[
		\Adm_Q(\Gamma) := \{ \rho \in \Adm (\Gamma) : \rho(e) = \rho(Q(e)), \, \forall e \in E \}
		\]
		as the set of $Q$-invariant admissible densities.

		\begin{proposition}\label{pro:symmetry}
			Assume that $\Gamma$ is a family of objects with weights $\si \in \R^E_{\geq 0}$ that is also $Q$-invariant. Then, for $ 1\leq p < \infty,$ 
			\[
			\Mod_{p,\sigma}(\Gamma) = \inf_{\rho \in \Adm_Q(\Gamma)} \, \cE_{p,\sigma}(\rho).
			\]
		\end{proposition}
		
		\begin{proof}
			Suppose $\rho$ is extremal for $\Mod_{p,\sigma}(\Gamma)$, 
			so that in particular $\rho \in \Adm(\Gamma)$. Define 
			$\rho_1 := \rho \circ Q$. Then, by \ref{eq:invariant}, for every 
			$\gamma_1 \in \Gamma$, there exists $\gamma_2 \in \Gamma$ 
			such that
			\[
			\ell_{\rho_1}(\gamma_1) = \sum_{e\in E} 
			N(\gamma_1, e)\rho_1(e) = \sum_{e\in E} 
			N(\gamma_2, Q(e))\rho(Q(e)) = \ell_{\rho}(\gamma_2) \geq 1.
			\]
			Therefore, $\rho_1 \in \Adm(\Gamma)$. Moreover, by \ref{eq:symmetry},
			\[
			\mathcal{E}_{p,\sigma}(\rho_1) = \sum_{e\in E} 
			\sigma(e)|\rho_1(e)|^p
			= \sum_{e\in E} \sigma(Q(e))|\rho(Q(e))|^p
			= \mathcal{E}_{p,\sigma}(\rho).
			\]
			
			Define $\rho_2 := \frac{\rho + \rho_1}{2}$. Then, 
			$\rho_2 \in \Adm(\Gamma)$. Also, since $Q$ is an involution, 
			we have \[\rho_2 \circ Q =  \frac{\rho \circ Q+ \rho_1\circ Q}{2} =  \frac{\rho_1+ \rho}{2} =  \rho_2.\] So, 
			$\rho_2 \in \Adm_Q(\Gamma)$. 
			
			By convexity,
			\[
			\mathcal{E}_{p,\sigma}(\rho_2) \leq 
			\frac{\mathcal{E}_{p,\sigma}(\rho) + 
				\mathcal{E}_{p,\sigma}(\rho_1)}{2}
			= \Mod_{p,\sigma}(\Gamma).
			\]
			Therefore,
			\[
			\inf_{\rho \in \Adm_Q(\Gamma)} 
			\mathcal{E}_{p,\sigma}(\rho) \leq 
			\Mod_{p,\sigma}(\Gamma).
			\]
			The other direction follows from the fact that 
			$\Adm_Q(\Gamma) \subset \Adm(\Gamma)$.
			
		\end{proof}

		Let $e$ and $f$ be two distinct elements in $E$. We say 
		that $e$ and $f$ are {\it symmetric} if $\Ga$ is 
		$Q$-invariant for some symmetry $Q$ and $Q(e) = f$. 
		Assume that $P =\{ E_1,E_2,\dots, E_k\}$ is a partition 
		of the ground set $E$ such that all elements in each 
		class of $P$ are symmetric to each other. Then, we 
		define a ground set $E_P$ as the set of classes of $P$, 
		and we enumerate $E_P$ as $\{q_1,\dots,q_k\}$. 
		Then, $|E_P| = |P|$. Let $\Ga_P$ be the family of all 
		vectors $\ga_P \in\R_{\geq 0}^{E_P}$ associated with 
		$\ga \in \Ga$ and defined as  
		\[
		\ga_P(q_i) = \sum\limits_{e \in E_i}\cN(\ga,e)
		\]
		for $i = 1,\dots,k$. 
		We also define weights $\si_P \in \R_{\geq 0}^{E_P}$ 
		with 
		\[
		\si_P(q_i) = \si(e)|E_i|
		\]
		for some $e \in E_i$.
		Using the symmetry property in Proposition \ref{pro:symmetry} for modulus, we obtain that 	
		\[
		\Mod_{p,\sigma}(\Gamma) = \text{Mod}_{p,\sigma_P}(\Gamma_P).
		\]
		Moreover, let $\rho^* \in \R^E_{\geq 0}$ be an optimal density for $\Mod_{p,\sigma}(\Gamma)$. Then, a density $\tilde{\rho}^*$, defined as $\tilde{\rho}^*(q_i) = \rho^*(e)$ for some $e \in E_i$, is optimal for $\Mod_{p,\sigma_P}(\Gamma_P)$.
	Recall that the optimal density  $\eta^*$ for the Fulkerson dual problem of $\Mod_{p,\sigma}(\Gamma)$  satisfies
		
		\[\eta^*(e) = \frac{\si(e)\rho^*(e)^{p-1}}{\text{Mod}_{p,\sigma}(\Gamma)}.\] It follows that the optimal density $\tilde{\eta}^*$ for the Fulkeron dual problem of $\Mod_{p,\sigma_P}(\Gamma_P)$ satisfies 
		\[\tilde{\eta}^*(q_i) = \frac{\si_P(q_i)\tilde{\rho}^*(q_i)^{p-1}}{\text{Mod}_{p,\sigma_P}(\Gamma_P)} = |E_i|\eta^*(e) \text{ for some } e \in E_i.\]
		
		In the following section, we will use these properties to study the modulus of the multi-tree family of hypergraphs.

		\subsection{Modulus for the multi-tree family}
		Let $H = (V,E)$ be a connected hypergraph that is not necessarily partition-connected. Let $\Omega(H)$ be the multi-tree family of $H$. Let $t = |V|$, we recall the hypergraph $ H^t = (V, E^t) $ where $E^t$ is obtained by replacing each hyperedge $e$ in $E$ by a set of $t$ parallel hyperedges.  By Remark \ref{re:1-1}, there is a one to one correspondence between two families $\Omega(H)$ and $\Ga(H^t)$.
		
		For the hypergraph $H^t$, consider a partition $P$ of $E$ where each class of $P$ is the set of $t$ parallel hyperedges in $E(H^t)$ corresponding to each hyperedge $e$ in $E(H)$. The hyperedges in each class are symmetric with respect to the family $\Ga(H^t)$. Hence, the set $E(H)$ can be identified with the set of all classes of $P$ and $\Omega(H)$ is the family of usage vectors in $\R^{E(H)}_{\geq 0} $ associated with the family $\Ga(H^t)$ generated by the partition $P$.
		
		First, let us analyze 1-modulus of the multi-tree family. By the symmetry property for modulus in Section \ref{sec:sym}, it follows that \[t\Mod_1(\Omega) = \Mod_1(\Ga(H^t)).\] 
		By the relationship between modulus of hypertrees and strength of hypergraphs in Theorem \ref{mod-strength}, we have \[\Mod_1(\Ga(H^t)) = S(H^t) =  tS(H).\] Therefore, we obtain that $\Mod_1(\Omega) = S(H)$, which also was shown in (\ref{eq:omega-s}).
		
		Next, we apply the symmetry property of modulus in Section \ref{sec:sym} to deduce results of the 2-modulus of $\Omega(H)$ from the 2-modulus of $\Ga(H^t)$. 
		
		\begin{theorem}\label{thm:main2}
			Let $H$ be a connected hypergraph. Let $\Omega=\Omega(H)$ be the multi-tree family of $H$. Let $\tilde{\eta}^*$ be the optimal density of $\Mod_{2}(\widehat{\Omega})$. Let $E_{max}$ be the set of hyperedges $e$ where $\tilde{\eta}^*(e)$ attains the maximum value $\tilde{\eta}^*_{max}$ of $\tilde{\eta}^*$, and let $E_{min}$ be the set of hyperedges $e$ where $\tilde{\eta}^*(e)$ attains the minimum value $\tilde{\eta}^*_{min}$ of $\tilde{\eta}^*$. Let $S(H)$ and $D(H)$ be the strength and fractional arboricity of $H$. 
			Then, \[ S(H) = \frac{1}{\tilde{\eta}^*_{max}}, \quad D(H) = \frac{1}{\tilde{\eta}^*_{min}}.\] Moreover, $E_{max}$ and $E_{min}$ are optimal for the strength problem and fractional arboricity problem, respectively.
		\end{theorem}
		\begin{proof}
			Let $t = |V|$, we recall the hypergraph $ H^t = (V, E^t) $ where $E^t$ is obtained by replacing each hyperedge $e$ in $E$ by a set of $t$ parallel hyperedges. For the hypergraph $H^t$, consider a partition $P$ of $E(H^t)$ where each class of $P$ is the set of $t$ parallel hyperedges in $E(H^t)$ corresponding to each hyperedge $e$ in $E(H)$. Note that every class of  $P$ has size $t$, by the symmetry property of modulus, we obtain
			\[
			t\Mod_{2}(\Omega(H)) = \Mod_{2}(\Ga(H^t)).
			\]
			Moreover, let $\eta^*$ be the optimal density for $\Mod_{2}(\widehat{\Gamma}(H^t))$. It follows that the optimal density $\tilde{\eta}^*$ of $\Mod_{2}(\widehat{\Omega})$ satisfies \[\tilde{\eta}^*(e) =  t\eta^*(e_1),\] for every $e \in E$ and for some copy $e_1$ in $E(H^t)$ of $e$.
			
			Now, we apply the modulus for the hypertree family $\Ga(H^t)$ as follows. 
			Apply Theorem \ref{main1} to $\Ga(H^t)$ and noting that $E_{max}$ corresponds to hyperedges in $H^t$ that attain the maximum value of $\eta^*$, we obtain  \[E_{max} = \delta_{H}(\PP),\]  for some feasible partition $\PP = \{V_1,\dots,V_k\}$ of $V(H)$ in the hypergraph $H$. Moreover, \[tS(H) =S(H^t) = \frac{1}{\eta^*_{max}} = \frac{t}{\tilde{\eta}^*_{max}}.\] Hence, \[S(H) = 1/\tilde{\eta}^*_{max},\] and $E_{max}$ is optimal for the strength problem $S(H)$.
			Similarly, by Theorem \ref{thm:mod:f}, we also have \[D(H) = 1/\tilde{\eta}^*_{min},\] and $E_{min}$ is optimal for the fractional arboricity problem $D(H)$.
		\end{proof}

		\begin{proof}[Proof for Theorem \ref{thm:c-hom}] As a corollary of Theorem \ref{thm:main2}, the  optimal density $\tilde{\eta}^*$ of $\Mod_{2}(\widehat{\Omega})$ is a constant vector if and only if the strength of $H$ coincides with the fractional arboricity of $H$.
		\end{proof}
		
		We can also deduce a serial rule for 2-modulus of $\Omega(H)$ from the serial rule for 2-modulus of $\Ga(H^t)$ in Theorem \ref{thm:emax-serial}, this is deduced in the same manner as the proof of Theorem \ref{thm:main2}. 

		\begin{theorem}\label{thm:omega-emax-serial}
			Let $H$ be a connected hypergraph. Let $\Omega =\Omega(H)$ be the multi-tree family of $H$. Assume that the optimal $\tilde{\eta}^*$ of $\Mod_2(\widehat{\Omega})$ is not a constant vector, and let $E_{max}$ be the set of hyperedges $e$ where $\tilde{\eta}^*(e)$ attains the maximum value $\tilde{\eta}^*_{max}$ of $\tilde{\eta}^*$.  Let $H/E_{max}$ be the hypergraph obtained from $H$ by contracting $E_{max}$. Let the hypergraph $H[E- E_{max}]$ have $k$ connected components $L_1,\dots,L_k$ with $L_i=(V_i,C_i)$, for $i=1,\dots,k$. Then, 
			\bi 
			\item[ (i)] The modulus problem $\Mod_2 (\widehat{\Omega}(H))$ splits as follows:\begin{equation}
				\Mod_2 (\widehat{\Omega}(H)) = \Mod_2 (\widehat{\Omega}(L_1)) + \Mod_2 (\widehat{\Omega}(L_2)) +\dots + \Mod_2 (\widehat{\Omega}(L_k))+ \Mod_2 (\widehat{\Omega}(H/E_{max}));
			\end{equation}
			\item[ (ii)] The restriction of $\tilde{\eta}^*$ onto $C_i$ is optimal for $\Mod_2(\widehat{\Omega}(L_i)) $ for $i = 1,2,\dots,k$, and the  restriction of $\tilde{\eta}^*$ onto $E_{max}$  is optimal for $\Mod_2(\widehat{\Omega}(H/E_{max}))$. 
			\ei
		\end{theorem}

		\textbf{Hypergraph decomposition process:} 
		Let $H$ be a connected hypergraph that is not homogeneous. 
		The hypergraph decomposition process for the 
		partition-connected hypergraph $H^t$ corresponds to a 
		hypergraph decomposition process for the connected 
		hypergraph $H$ by scaling. 
		This yields a hierarchical structure of $H$.

		\begin{proof}[Proof for Theorem \ref{thm:mod-mod}]
			Theorem \ref{thm:main2} and the hypergraph decomposition process of $H$ imply that if $H$ is partition-connected, $\Mod_{2}(\widehat{\Omega})$ and $\Mod_{2}(\widehat{\Ga})$ have the same optimal density and they reveal the same hierarchical structure of $H$.
		\end{proof}
		
		\begin{remark}
			If $H$ is not partition-connected, the hierarchical structure of $H$ is not necessarily related to the hierarchical structure of $M(H)$. Here is an example: $H=(V=\{v_1, v_2, v_3,v_4\}, E=\{e_1,e_2\})$ where
			$e_1=\{v_1, v_2, v_3\}$, $e_2=\{ v_3,v_4\}$. This remark is similar to Remark \ref{re:s-neq-s}.
		\end{remark} 
		

		\bibliographystyle{acm}
		\bibliography{paper.bib}

\begin{thebibliography}{10}

\bibitem{modulus}
{\sc Albin, N., Brunner, M., Perez, R., Poggi-Corradini, P., and Wiens, N.}
\newblock Modulus on graphs as a generalization of standard graph theoretic
  quantities.
\newblock {\em Conform. Geom. Dyn. 19\/} (2015), 298--317.

\bibitem{pietroblocking}
{\sc Albin, N., Clemens, J., Fernando, N., and Poggi-Corradini, P.}
\newblock Blocking duality for {$p$}-modulus on networks and applications.
\newblock {\em Ann. Mat. Pura Appl. (4) 198}, 3 (2019), 973--999.

\bibitem{pietrofairest}
{\sc Albin, N., Clemens, J., Hoare, D., Poggi-Corradini, P., Sit, B., and
  Tymochko, S.}
\newblock Fairest edge usage and minimum expected overlap for random spanning
  trees.
\newblock {\em Discrete Math. 344}, 5 (2021), Paper No. 112282, 24.

\bibitem{pietrominimal}
{\sc Albin, N., and Poggi-Corradini, P.}
\newblock Minimal subfamilies and the probabilistic interpretation for modulus
  on graphs.
\newblock {\em J. Anal. 24}, 2 (2016), 183--208.

\bibitem{connection15}
{\sc Bahmanian, M.~A., and Sajna, M.}
\newblock Connection and separation in hypergraphs.
\newblock {\em Theory and Applications of Graphs 2}, 2 (2015), 5.

\bibitem{ba23al}
{\sc Baiou, M., and Barahona, F.}
\newblock On some algorithmic aspects of hypergraphic matroids.
\newblock {\em Discrete Mathematics 346}, 2 (2023), 113222.

\bibitem{chopraon}
{\sc Chopra, S.}
\newblock On the spanning tree polyhedron.
\newblock {\em Oper. Res. Lett. 8}, 1 (1989), 25--29.

\bibitem{Frank}
{\sc Frank, A., Kir{\'a}ly, T., and Kriesell, M.}
\newblock On decomposing a hypergraph into k connected sub-hypergraphs.
\newblock {\em Discrete Applied Mathematics 131}, 2 (2003), 373--383.

\bibitem{fulkersonblocking}
{\sc Fulkerson, D.~R.}
\newblock Blocking polyhedra.
\newblock In {\em Graph {T}heory and its {A}pplications ({P}roc. {A}dvanced
  {S}em., {M}ath. {R}esearch {C}enter, {U}niv. of {W}isconsin, {M}adison,
  {W}is., 1969)\/} (1970), Academic Press, New York, pp.~93--112.

\bibitem{fulkersonanti}
{\sc Fulkerson, D.~R.}
\newblock Blocking and anti-blocking pairs of polyhedra.
\newblock {\em Math. Programming 1\/} (1971), 168--194.

\bibitem{Lo}
{\sc Lorea, M.}
\newblock Hypergraphes et matroides.
\newblock {\em Cahiers Centre Etudes Rech. Oper. 17\/} (1975), 289--291.

\bibitem{lovasz70}
{\sc Lov{\'a}sz, L.}
\newblock A generalization of konig's theorem.
\newblock {\em Acta Mathematica Academiae Scientiarum Hungarica 21\/} (1970),
  443--446.

\bibitem{nash1964decomposition}
{\sc Nash-Williams, C. S.~J.}
\newblock Decomposition of finite graphs into forests.
\newblock {\em Journal of the London Mathematical Society 1}, 1 (1964), 12--12.

\bibitem{edge-disjoint}
{\sc Nash-Williams, C. S. J.~A.}
\newblock Edge-disjoint spanning trees of finite graphs.
\newblock {\em J. London Math. Soc. 36\/} (1961), 445--450.

\bibitem{huyfulkerson}
{\sc Truong, H., and Poggi-Corradini, P.}
\newblock Fulkerson duality for modulus of spanning trees and partitions.
\newblock {\em arXiv preprint arXiv:2306.15984\/} (2023).

\bibitem{truong2024reinforcement}
{\sc Truong, H., and Poggi-Corradini, P.}
\newblock Matroid reinforcement and sparsification.
\newblock {\em arXiv preprint arXiv:2408.00173\/} (2024).

\bibitem{basemodulus}
{\sc Truong, H., and Poggi-Corradini, P.}
\newblock Modulus for bases of matroids.
\newblock {\em Discrete Math. 348}, 5 (2025), Paper No. 114395, 26.

\bibitem{on}
{\sc Tutte, W.~T.}
\newblock On the problem of decomposing a graph into {$n$} connected factors.
\newblock {\em J. London Math. Soc. 36\/} (1961), 221--230.

\end{thebibliography}
		\def\cprime{$'$}
		\nocite{*}

	\end{document}